\documentclass[10pt]{amsart}
\usepackage[margin=1.4in,letterpaper,portrait]{geometry}
\usepackage{latexsym,amssymb,verbatim,multirow,graphicx,epsfig,enumerate, paralist,xypic}
\usepackage{xcolor, tikz}

\theoremstyle{plain}
   \newtheorem{theorem}{Theorem}[section]

   \newtheorem{lemma}[theorem]{Lemma}

   \newtheorem*{theorem*}{Theorem}
   \newtheorem*{main}{Main Theorem}
\theoremstyle{definition}

   \newtheorem{remark}[theorem]{Remark}

\numberwithin{equation}{section}

\newcommand\Symm{\mathfrak{S}}

\newcommand{\boldf}{\mathbf{f}}

\newcommand{\type}[1]{\mathrm{#1}}

\newcommand\CC{{\mathbb{C}}}
\newcommand\ZZ{{\mathbb{Z}}}

\newcommand\RR{{\mathbb{R}}}

\newcommand{\BB}{\mathcal{B}}
\renewcommand{\tt}{\mathbf{t}}

\newcommand{\til}{\widetilde}
\newcommand{\id}{\varepsilon}

\begin{document}

\title[The Hurwitz action in complex reflection groups]{A note on the Hurwitz action on reflection factorizations of Coxeter elements in complex reflection groups}
\author{Joel Brewster Lewis}
\date{\today}

\begin{abstract}
We show that the Hurwitz action is ``as transitive as possible'' on reflection factorizations of Coxeter elements in the well generated complex reflection groups $G(d, 1, n)$ (the group of $d$-colored permutations) and $G(d, d, n)$.
\end{abstract}

\maketitle

\section{Introduction}
Given a group $G$ with a generating set $T$ that is closed under conjugation, the \emph{braid group} $\BB_m := \langle \sigma_1, \ldots, \sigma_{m - 1} \mid \sigma_i \sigma_{i + 1} \sigma_i = \sigma_{i + 1} \sigma_i \sigma_{i + 1} \rangle$ on $m$ strands acts on $T^m$ via
\[
\begin{array}{rccl}
\big(t_1, \ldots, t_{i - 1}, \quad & t_i, \quad & t_{i + 1}, &\quad t_{i + 2}, \ldots, t_m\big) \overset{\sigma_i}{\longmapsto} \\[2pt]
\big(t_1, \ldots, t_{i - 1}, \quad &  t_{i + 1}, \quad & (t_i)^{t_{i + 1}}, &\quad t_{i + 2}, \ldots, t_m\big),
\end{array}
\]
where $ (t_i)^{t_{i + 1}} :=   t_{i + 1}^{-1} \cdot t_i \cdot t_{i + 1}$ represents conjugation.  The individual moves $\sigma_i$ are called \emph{Hurwitz moves}, and the entire action is called the \emph{Hurwitz action}.  Given a tuple $\tt \in T^m$, the product of the elements of $\sigma_i(\tt)$ is equal to the product of the elements of $\tt$, so that the Hurwitz action may be viewed as an action on \emph{$T$-factorizations} of a given element $c = t_1 \cdots t_m$ in $G$.  Moreover, the Hurwitz action clearly preserves the set of conjugacy classes of the tuple $\tt$ on which it acts.  In general, there may be multiple orbits of factorizations of a given element $c$ with fixed tuple of conjugacy classes under the Hurwitz action, but in \cite[Conj.~6.3]{LReiner}, it was conjectured that if $G$ is a \emph{well generated complex reflection group} 
and $c$ is a \emph{Coxeter element} in $G$, then the multiset of conjugacy classes of the factors in a factorization completely determines the Hurwitz orbit to which it belongs.  The purpose of this note is to prove the conjecture for the two combinatorial families of well generated complex reflection groups.
\begin{main}
Fix integers $n \geq 2$, $d \geq 2$, let $G$ be either of the complex reflection groups $G(d, 1, n)$ and $G(d, d, n)$, and let $c$ be a Coxeter element in $G$.  Then two reflection factorizations of $c$ are in the same Hurwitz orbit if and only if they have the same multiset of conjugacy classes of reflections.
\end{main}
The remainder of this document is structured as follows:
In Section~\ref{sec:background}, we give the mathematical context and background for the present work.  In Section~\ref{sec:CRG}, we introduce the main objects of our study: the infinite families of complex reflection groups and their generic covers.  In Section~\ref{sec:proof}, we prove the main result.  Finally, in Section~\ref{sec:exceptionals}, we make some remarks about the exceptional complex reflection groups (those that do not belong to the infinite family).

\section{Context and background}
\label{sec:background}

%In this section, we present the background and motivation for the present work.  
We begin with a review of some previous work on the Hurwitz action, much of which takes place in the context of reflection groups (either complex or Coxeter).

The original motivations for the study of the Hurwitz action are geometric: it was introduced by Hurwitz \cite{Hurwitz} as part of his study of covering surfaces of the Riemann sphere.  In that context, the group $G$ under consideration is the symmetric group $\Symm_n$ and the allowable factors $T$ are the transpositions, which case has been the subject of much further study (e.g., \cite{Kluitmann, BenItzhak-Teicher}).  More generally, as part of his study of the geometry of hyperplane arrangements, Bessis considered the case that $G$ is a complex reflection group and $T$ is the set of reflections in $G$.  (These notions are defined in the next section.)
\begin{theorem}[{Bessis \cite[Prop.~7.6]{BessisKpi1}}]
\label{thm:bessis}
Let $G$ be a well generated complex reflection group and $c$ a Coxeter element in $G$.  Then the Hurwitz action is transitive on minimum-length reflection factorizations of $G$.
\end{theorem}
This result was proved on a case-by-case basis, using the classification of complex reflection groups.  The corresponding result for arbitrary Coxeter groups was proved by Igusa and Schiffler \cite{IgusaSchiffler}, and given a short, elegant proof in \cite{BDSW}.  

In the case in finite Coxeter groups (synonymously, finite real reflection groups), much is known about the structure of the Hurwitz orbits of reflection factorizations.  The elements with the property that the Hurwitz action is transitive on their minimum-length reflection factorizations are completely classified \cite{BGRW}.  In \cite{LReiner}, the present author and Reiner showed that two reflection factorizations (of any length) of a Coxeter element in a finite Coxeter group belong to the same Hurwitz orbit if and only if they share the same multiset of conjugacy classes.  This result was recently extended by Wegener and Yahiatene.
\begin{theorem}[{Wegener--Yahiatene \cite[Thm.~1.2]{WY}}]
\label{thm:WY}
Let $G$ be an arbitrary Coxeter group and $c$ a Coxeter element in $G$. Then two reflection factorizations of $c$ lie in the same Hurwitz orbit if and only if they share the same multiset of conjugacy classes.
\end{theorem}
When not restricted to reflection factorizations, the Hurwitz action has been completely analyzed for dihedral groups and some related families \cite{HouQuaternion, SiaDihedral, BergerDihedral}.

Though the present work is not focused on enumeration, we should mention that reflection factorizations of Coxeter elements were enumerated by Chapuy and Stump, with a beautiful uniform formula, proved case-by-case \cite{ChapuyStump}.  The Chapuy--Stump formula was proved uniformly for Weyl groups by Michel \cite{Michel}, and was further refined by DelMas, Hameister, and Reiner \cite{DMHR}.  Subsequently, these enumerations have been generalized to all regular elements and given a uniform proof by Douvropoulos \cite{TheoEnumeration}.

\section{Complex reflection groups}
\label{sec:CRG}

\subsection{Basic definitions, classification, Coxeter elements}

In this section, we give relevant background on complex reflection groups.  For a more thorough treatment (though not necessarily in the same notation), see \cite[Chs.~2, 11, 12]{LehrerTaylor}

Given a finite-dimensional complex vector space $V$, a \emph{reflection} is a linear transformation $t: V \to V$ whose fixed space $\ker(t - 1)$ is a hyperplane (i.e., has codimension $1$), and a finite subgroup $G$ of $GL(V)$ is called a \emph{complex reflection group} if $G$ is generated by its subset $T$ of reflections.\footnote{Often in the literature it is required \emph{a priori} that reflections have finite order, or that $G$ be a group of unitary transformations.  We omit these conditions because they do not affect the resulting classification.}  For example, if $d$, $e$ and $n$ are positive integers, then the group 
\[
G(de, e, n) := \left\{ \begin{array}{c} n \times n \textrm{ monomial matrices whose nonzero entries are} \\
\textrm{$(de)$th roots of unity with product a $d$th root of unity} \end{array} \right\}
\]
is a complex reflection group acting on $\CC^n$: writing $\omega = \exp(2\pi i/ de)$ for the primitive $(de)$th root of unity, the reflections are the transposition-like reflections
\begin{equation}
\label{eq:transposition}
\left[
\begin{array}{ccccccccccc}
1 &&&&&&&&&& \\
& \ddots &&&&&&&& \\
&& 1 &&&&&&&& \\
&&& &&&& \omega^k &&& \\
&&&& 1 &&&&&& \\
&&&&& \ddots &&&&& \\
&&&&&& 1 &&&& \\
&&& \overline{\omega}^{k} &&&&&&& \\
&&&&&&&& 1 && \\
&&&&&&&&& \ddots & \\
&&&&&&&&&& 1
\end{array}
\right]
\end{equation}
of order $2$, fixing the hyperplane $x_i = \omega^k x_j$, and the diagonal reflections
\begin{equation}
\label{eq:diagonal}
\left[
\begin{array}{ccccccc}
1 &&&&&& \\
& \ddots &&&&& \\
&& 1 &&&& \\
&&& \omega^{ek} &&& \\
&&&& 1 && \\
&&&&& \ddots & \\
&&&&&& 1
\end{array}
\right]
\end{equation}
of various orders, fixing the hyperplane $x_i = 0$.  It is natural to represent such groups combinatorially: the group $G(de, 1, n)$ is isomorphic to the \emph{wreath product} $\ZZ/de\ZZ \wr \Symm_n$ of a cyclic group with the symmetric group $\Symm_n$.  The elements of $\ZZ/de\ZZ \wr \Symm_n$ are pairs $[w; a]$ with $w \in \Symm_n$ and $a = (a_1, \ldots, a_n) \in (\ZZ/de\ZZ)^n$, with product given by 
\[
[w; a] \cdot [u; b] = [wu; u(a) + b], \quad \textrm{ where } \quad u(a) := \left( a_{u(1)}, \ldots, a_{u(n)} \right).
\]
Under this isomorphism, $w$ is the underlying permutation of the monomial matrix corresponding to $[w; a]$, while $a_k$ is the exponent to which $\omega = \exp\left(2\pi i / de\right)$ appears in the nonzero entry in the $k$th column.    For $S \subseteq \{1, \ldots, n\}$, we say that $\sum_{k \in S} a_k$ is the \emph{weight} of $S$; this notion will come up particularly when the elements of $S$ form a cycle in $w$.  When $S = \{1, \ldots, n\}$, we call $a_1 + \ldots + a_n$ the \emph{weight} of the element $[w; a]$.

We denote by $\id$ the identity permutation in the symmetric group $\Symm_n$, so the diagonal reflection in \eqref{eq:diagonal} corresponds to $[ \id; (0, \ldots, 0, ek, 0, \ldots, 0)]$.  In the case of the transposition-like reflections, we condense the wreath-product notation even further and write $[(i \; j); k]$ for the reflection in \eqref{eq:transposition}, rather than the longer $[(i \; j); (0, \ldots, 0, -k, 0, \ldots, 0, k, 0, \ldots, 0)]$.  

Complex reflection groups were classified by Shephard and Todd \cite{ShephardTodd}: every \emph{irreducible} complex reflection group is either isomorphic to some $G(de, e, n)$ or to one of $34$ exceptional examples, and every complex reflection group is a direct sum of irreducibles.  In the present work we focus on the groups $G(de, e, n)$, but we briefly discuss the exceptional groups in Section~\ref{sec:exceptionals}.

An element $g$ of a complex reflection group is called \emph{regular} if it has an eigenvector that does not lie on any of the fixed hyperplanes of any of the reflections in $G$.  A \emph{Coxeter element} in $G$ is a regular element of multiplicative order $h := \frac{|T| + |A|}{n}$, where $T$ is the set of reflections in $G$, $A$ is the set of reflecting hyperplanes, and $n$ is the dimension of the space on which $G$ acts.\footnote{In different sources, one finds other, not-necessarily equivalent definitions of Coxeter elements.  For example, in \cite{BessisKpi1, TheoEnumeration}, Coxeter elements are taken to be those for which the eigenvector can be chosen with eigenvalue $\exp(2 \pi i/h)$.  For more discussion, see \cite{RRS}.}  Not every complex reflection group contains Coxeter elements; those that do are called \emph{well generated} in some sources and \emph{duality groups} in others.  In the infinite family $G(de, e, n)$, the well generated groups are precisely $G(d, 1, n)$ and $G(d, d, n)$.  For example, the elements
\begin{equation}
\label{eq:std cox}
\begin{bmatrix}
&&&&\omega\\
1&&&&\\
&1&&&\\
&&\ddots&&\\
&&&1&
\end{bmatrix} 
\in G(d, 1, n)
\quad
\textrm{ and } 
\quad
\begin{bmatrix}
&&&\omega&\\
1&&&&\\
&\ddots&&&\\
&&1&& \\
&&&& \overline{\omega}
\end{bmatrix}
\in G(d, d, n) 
\end{equation}
are Coxeter elements in their respective groups, where $\omega = \exp(2\pi i/d)$. 

There are $d$ conjugacy classes of reflections in $G(d, 1, n)$: all the transposition-like reflections are conjugate to each other, while the diagonal reflections fall into $d - 1$ classes depending on their weight.  When $n = 2$, the group $G(d, d, n)$ is the dihedral group of order $2d$; thus, it has one conjugacy class of reflections if $d$ is odd and two if $d$ is even.  For $n \geq 3$, all reflections in $G(d, d, n)$ are conjugate to each other.

\begin{remark}
\label{rmk:one suffices}
In a Weyl group, all Coxeter elements are conjugate.  For other complex reflection groups, this is not necessarily the case; however, by \cite[Prop.~1.4]{RRS}, if $c$ and $c'$ are Coxeter elements in a complex reflection group $G$, then there exists a group automorphism of $G$ that sends $c$ to $c'$ and sends reflections to reflections.  Consequently, in order to prove the Main Theorem for all Coxeter elements it suffices to prove it for just one.
\end{remark}

\subsection{Generic covers}
\label{sec:generic}

As described in the previous section, the group $G(d, 1, n)$ is isomorphic to the wreath product $\ZZ/d\ZZ \wr \Symm_n$ of the cyclic group $\ZZ/d\ZZ$ by the symmetric group $\Symm_n$.  Consequently, for each $d$ there is a natural projection $\pi_d: G(\infty, 1, n) \to G(d, 1, n)$ from the infinite group $G(\infty, 1, n) := \ZZ \wr \Symm_n$ onto $G(d, 1, n)$ that reduces the weight vector $a$ in the element $[w; a]$ modulo $d$.  Moreover, this covering is compatible with the reflection group structure, in the following sense: we may view the elements of $G(\infty, 1, n)$ as monomial matrices whose nonzero entries are integer powers of a formal variable $x$, acting on the vector space $K^n$ where $K$ is an algebraically closed field containing $\CC$ and the formal variable $x$ (e.g., one could take $K$ to be $\overline{\CC(x)}$, or to be the Puiseux series in $x$).  The reflections are again the elements that fix a hyperplane, and again come in two families: the transposition-like reflections $[(i \; j); k] := [ (i \; j); (0, \dots, 0, -k, 0, \dots, 0, k, 0, \dots, 0)]$ for $i, j \in \{1, \ldots, n\}$, $k \in \ZZ$ and the diagonal reflections $[\id; (0, \dots, 0, k, 0, \dots, 0)]$ for $k \in \ZZ \smallsetminus \{0\}$.  These reflections generate $G(\infty, 1, n)$, and every reflection $r$ in $G(d, 1, n)$ has in its fiber $\pi_d^{-1}(r)$ reflections of $G(\infty, 1, n)$.  The converse is not quite true: the image $\pi_d(\til{r})$ of a reflection $\til{r}$ in $G(\infty, 1, n)$ is a reflection in $G(d, 1, n)$ unless $\til{r}$ is diagonal and has weight divisible by $d$, in which case the projection is the identity.
We may extend the definition of regular element to this setting; the element 
\begin{equation}
\label{eq:generic cox d1n}
\til{c} := 
\begin{bmatrix}
&&&&x\\
1&&&&\\
&1&&&\\
&&\ddots&&\\
&&&1&
\end{bmatrix} 
\in G(\infty, 1, n)
\end{equation}
with (right) eigenvector $(x, \dots, x^{2/n}, x^{1/n})^T$ is one example.  For any $d$, the image of $\til{c}$ under $\pi_d$ is the Coxeter element for $G(d, 1, n)$ that appears in \eqref{eq:std cox}, and for these reasons we say that $\til{c}$ is the \emph{standard Coxeter element} in $G(\infty, 1, n)$.

Just as $G(d, 1, n)$ has a subgroup $G(d, d, n)$, the wreath product $G(\infty, 1, n)$ has a subgroup $G(\infty, \infty, n)$ consisting of all elements of weight $0$.  For each $d$, the projection map $\pi_d$ restricts to a covering $\pi_d: G(\infty, \infty, n) \to G(d, d, n)$.  Moreover, it was shown by Shi \cite[Thm.~2.3]{Shi-generic} that in fact $G(\infty, \infty, n)$ is isomorphic to the affine symmetric group, the Coxeter group of affine type A.  This group has Coxeter presentation 
\[
%G(\infty, \infty, n) \cong 
\left\langle s_1, \ldots, s_{n - 1}, s_n = s_0 \middle |
\begin{array}{ll} 
s_i^2 = 1 & \textrm{ for } i = 1, \ldots, n \\
s_i s_j = s_j s_i & \textrm{ if } i \neq j \pm 1 \\
s_i s_{i + 1} s_i = s_{i + 1} s_i s_{i + 1} & \textrm{ for } i = 0, \ldots, n - 1
\end{array}
\right\rangle,
\] 
where concretely we can take the generators to be $s_i = [ (i \; i + 1); 0]$ for $i = 1, \ldots, n - 1$ and $s_0 = s_n = [(1 \; n); 1]$.  In an arbitrary Coxeter group, the reflections are defined to be the conjugates of the generators $s_i$; in $G(\infty, \infty, n)$, these coincide exactly with the reflections in $G(\infty, 1, n)$ that belong to $G(\infty, \infty, n)$, namely, the transposition-like reflections $[(i \; j); k]$.  Under the projection $\pi_d$, these are mapped surjectively onto the reflections in $G(d, d, n)$.  

In a Coxeter group, one defines a Coxeter element to be a product of the simple generators in some order.  For example, 
\begin{equation}
\label{eq:generic cox ddn}
\til{c} := 
s_0 s_1 \cdots s_{n - 1} = 
\begin{bmatrix}
&&&x&\\
1&&&&\\
&\ddots&&&\\
&&1&& \\
&&&& x^{-1}
\end{bmatrix}
\end{equation}
is a Coxeter element in $G(\infty, \infty, n)$ in this sense.  The element $\til{c}$ is a regular element in $G(\infty, \infty, n)$, with eigenvector $(x, \ldots, x^{2/(n - 1)}, x^{1/(n - 1)}, 0)^T$, and its image under $\pi_d$ is the Coxeter element for $G(d, d, n)$ that appears in \eqref{eq:std cox}; for these reasons we say that $\til{c}$ is the \emph{standard Coxeter element} in $G(\infty, \infty, n)$.

\begin{remark}
\label{rmk:LR}
When $n = 2$, the affine symmetric group $G(\infty, \infty, 2)$ is the infinite dihedral group, consisting of the isometries of the real line $\RR$ that preserve the integer lattice $\ZZ$.  (The standard Coxeter element in this case is translation by $1$.) This group was already considered in \cite{LReiner}, where it was used in the proof of the Main Theorem for finite real reflection groups.  In particular, \cite[\S\S 4.1, 4.3]{LReiner} imply that any reflection factorization $(t_1, \ldots, t_{2k + 1})$ of a reflection $t$ in $G(\infty, \infty, 2)$ is in the same Hurwitz orbit as a factorization of the form $(t'_1, t'_1, t'_3, t'_3, \ldots, t'_{2k - 1}, t'_{2k - 1}, t)$.
\end{remark}

\section{Proof of the main result}
\label{sec:proof}

In this section, we prove the Main Theorem.  The first step is the following lemma, which allows passing the relevant questions about $G(d, 1, n)$ and $G(d, d, n)$ to their generic covers.

\begin{lemma}
\label{lem:lifting}
Let $G_d$ be one of the groups $G(d, 1, n)$ and $G(d, d, n)$, and let $G_\infty$ be its generic cover.  Let $c \in G_d$ be the Coxeter element from \eqref{eq:std cox} and let $\til{c} \in G_\infty$ be the standard Coxeter element for $G_\infty$ that appears in \eqref{eq:generic cox d1n} or \eqref{eq:generic cox ddn}.  Then for any reflection factorization $(t_1, \ldots, t_k)$ of $c$ in $G_d$, there exists a reflection factorization $(\til{t}_1, \ldots, \til{t}_k)$ of $\til{c}$ in $G_\infty$ such that $\pi_d\left( \til{t}_i \right) = t_i$ for all $i$.
\end{lemma}
\begin{proof}
The first step is to build a reflection factorization of an element in $G_\infty$ that is ``close to'' $\til{c}$.  

If $G_d = G(d, 1, n)$, then every reflection factorization of $c$ contains diagonal reflections (since every product of transposition-like reflections belongs to $G(d, d, n)$, which $c$ does not); choose $i$ to be the smallest index such that $t_i$ is a diagonal reflection.  If instead $G_d = G(d, d, n)$, then every reflection factorization of $c$ contains a reflection whose underlying permutation is $(a, n)$ for some $a \in \{1, \ldots, n - 1\}$ (since all other reflections fix the standard basis vector $e_n$, which $c$ does not); choose $i$ to be the smallest index such that $t_i$ is such a reflection.  For each index $j \in \{1, \ldots, k\} \smallsetminus \{ i\}$, choose $\til{t}_j$ to be an arbitrary reflection in $\pi_d^{-1}(t_j)$.  Then choose $\til{t}_i$ as follows: if $G_d = G(d, 1, n)$, take $\til{t}_i$ to be the unique reflection in $\pi_d^{-1}(t_i)$ such that the weight of $\til{c'} := \til{t}_1 \cdots \til{t}_k$ is equal to $1$; if $G_d = G(d, d, n)$,  take $\til{t}_i$ to be the unique reflection in $\pi_d^{-1}(t_i)$ such that the weight of $n$ in $\til{c'} := \til{t}_1 \cdots \til{t}_k$ is $-1$.

By construction, $\til{c}$ and $\til{c'}$ have the same underlying permutation, corresponding cycles of $\til{c}$ and $\til{c'}$ have the same weight, and $\pi_d(\til{c}) = \pi_d(\til{c'}) = c$.  Thus we can write $\til{c} = [w; a]$ and $\til{c'} = [w; a']$ and we have $a - a' = d \cdot (b_1, \ldots, b_n)$ for some integers $b_1, \ldots, b_n$ such that $b_1 + \ldots + b_n = 0$; and moreover if $G_d = G(d, d, n)$ then also $b_n = 0$.  For $i = 1, \ldots, n$, define $b'_i = b_1 + \ldots + b_{i - 1}$ (so in particular $b'_1 = 0$), and let $\delta = [ \id; d\cdot (b'_1, b'_2, \ldots, b'_n)] \in G_\infty$.  If $G_d = G(d, 1, n)$ then $w = (1\; 2\; \cdots \; n)$ and so
\begin{align*}
\delta \til{c} \delta^{-1} &  = [w; (a_1 + db'_2 - db'_1, a_2 + db'_3 - db'_2, \ldots, a_{n - 1} + db'_n - db'_{n - 1}, a_n + db'_1 - db'_n)] \\
 & = [w; (a_1 + db_1, a_2 + db_2, \ldots, a_n + db_n)].
\end{align*} 
If instead $G_d = G(d, d, n)$ then $w = (1\; 2\; \cdots \; n - 1)(n)$ and so
\begin{align*}
\delta \til{c} \delta^{-1} &  = [w; (a_1 + db'_2 - db'_1,  \dots, a_{n - 2} + db'_{n - 1} - db'_{n - 2}, \ldots, a_{n - 1} + db'_1 - db'_{n - 1}, a_n)] \\
& = [w; (a_1 + db_1,  \dots, a_{n - 2} + db_{n - 2}, \ldots, a_{n - 1} + db_{n - 1}, a_n)].
\end{align*}
In both cases, the result of the conjugation is $\til{c'}$.  Moreover, since $\pi_d(\delta)$ is the identity element in $G_d$, it follows that
\[
\left( (\delta^{-1} \til{t}_1 \delta), \; \ldots, \; (\delta^{-1} \til{t}_k \delta)\right)
\]
is the desired reflection factorization of $\til{c}$.
\end{proof}

\begin{remark}
In general, the question of whether reflection factorizations of an element $g$ in $G(d, 1, n)$ or $G(d, d, n)$ can be lifted to the generic cover can be subtle.  One problem is that \emph{reflection length} is not preserved by projection: if $\til{g} \in G_\infty$ has minimum-length reflection factorizations of length $k$, it may be that its projection $\pi_d(\til{g})$ has shorter factorizations, that consequently cannot be lifted to factorizations of $\til{g}$.  A trivial example of this problem is the reflection $\til{g} = [\id; (d, 0, 0)] \in G(\infty, 1, 3)$ (having reflection length $1$), whose image is the identity in $G(d, 1, 3)$ (having reflection length $0$): the empty factorization of the identity cannot be lifted to a factorization of $\til{g}$.

Even when reflection length is preserved, it may not be possible to lift all factorizations.  For example, the central element $g = [ \id; (1, 1, 1)] = -1$ in $G(2, 1, 3)$ is the image of $\til{g} = [\id; (1, 1, 1)]$ in $G(\infty, 1, 3)$ under the projection $\pi_2$.  Both elements have reflection length $3$ (for example, they can be factored as the product of three diagonal reflections), but the reflection factorization
\[
g = [(1 \; 2); (0, 0, 0)] \cdot [(1 \; 2); (1, 1, 0)] \cdot [\id; (0, 0, 1)]
\]
is not the projection of any reflection factorization of $\til{g}$.  For more on reflection length in complex reflection groups, see \cite{Shi-reflection-length, FosterGreenwood}.
\end{remark}

%It follows from the discussion in Section~\ref{sec:CRG} that it suffices to prove the statement for the single Coxeter elements (ref, ref) in the groups $G(\infty, 1, n)$ and $G(\infty, \infty, n)$.

To finish the proof of the Main Theorem, we handle the groups $G(d, d, n)$ and $G(d, 1, n)$ separately.

\subsection{The group $G(d, d, n)$}

If $n = 2$ then $G_d := G(d, d, n)$ is the dihedral group of order $2d$.  The Main Theorem was proved for this group (along with all other finite real reflection groups) in \cite{LReiner}.

Now suppose that $n > 2$.  In this case, the reflections in $G_d$ form a single conjugacy class, and so the statement to be proved is that for each Coxeter element $c$ in $G_d$, any two reflection factorizations of the same length belong to the same Hurwitz orbit.  Moreover, from the discussion in Remark~\ref{rmk:one suffices} we know that it suffices to prove the statement for the single Coxeter element $c := [(1 \; 2 \cdots n - 1)(n); (0, \ldots, 0, 1, -1)]$ in \eqref{eq:std cox}.  Let $\til{c} := [(1 \; 2 \cdots n - 1)(n); (0, \ldots, 0, 1, -1)]$ be the standard Coxeter element in $G_\infty := G(\infty, \infty, n)$.

Fix an integer $m$ and two length-$m$ reflection factorizations $\tt = (t_1, \ldots, t_m)$ and $\tt' = (t'_1, \ldots, t'_m)$ of $c$.  By Lemma~\ref{lem:lifting}, there exist reflection factorizations $\til{\tt} = \left( \til{t}_1, \ldots, \til{t}_m\right)$ and $\til{\tt}' = \left( \til{t}'_1, \ldots, \til{t}'_m\right)$ of $\til{c}$ in $G_\infty$ such that $\pi_d(\til{t}_i) = t_i$ and $\pi_d(\til{t}'_i) = t'_i$ for all $i \in \{1, \ldots, m\}$.  The reflections in $G_\infty$ form a single conjugacy class (the defining Coxeter relations can written $s_{i + 1} = (s_{i + 1} s_{i})^{-1} s_i (s_{i + 1} s_{i})$, so all $s_i$ are conjugate, and each reflection is conjugate to one of the $s_i$), hence by Theorem~\ref{thm:WY} of Wegener--Yahiatene, the factorizations $\til{\tt}$ and $\til{\tt}'$ belong to the same Hurwitz orbit.  Hurwitz moves clearly commute with the projection $\pi_d$, so the same braid $\beta$ that satisfies $\beta(\til{\tt}) = \beta(\til{\tt}')$ also satisfies $\beta(\tt) = \beta(\tt')$. This completes the proof.

\subsection{The group $G(d, 1, n)$}

Our proof for $G(d, 1, n)$ begins the same as for $G(d, d, n)$; however, because the generic cover $G_\infty:= G(\infty, 1, n)$ is not a Coxeter group, we cannot make use of Theorem~\ref{thm:WY}.  Consequently, we employ a more hands-on approach.

It follows from Remark~\ref{rmk:one suffices} that it suffices to prove the statement for the single Coxeter element $c := [(1 \; 2 \cdots n); (0, \ldots, 0, 1)]$ in $G(d, 1, n)$ shown in \eqref{eq:std cox}.  Let $\til{c} := [(1 \; 2 \cdots n); (0, \ldots, 0, 1)]$ be the standard Coxeter element in $G_\infty$.  Choose a reflection factorization $\tt = (t_1, \ldots, t_m)$ of $c$.  By Lemma~\ref{lem:lifting}, there exists a reflection factorization $\til{\tt} = (\til{t}_1, \ldots, \til{t}_m)$ of $\til{c}$ in $G_\infty$ such that $\pi_d(\til{t}_i) = t_i$ for all $i \in \{1, \ldots, m\}$.  The bulk of the proof is to produce a canonical representative of the Hurwitz orbit of $\til{\tt}$.

We may use Hurwitz moves to move all the diagonal reflections in $\til{\tt}$ before all the transposition-like reflections, and so without loss of generality we assume that $\til{t}_1, \ldots, \til{t}_k$ are diagonal and $\til{t}_{k + 1}, \ldots, \til{t}_m$ are transposition-like.  Apply the projection $\pi_1 : G_\infty \to \Symm_n$; denoting $\widehat{t}_i = \pi_1(\til{t}_i)$, we have that $\widehat{\tt} := (\widehat{t}_{k + 1}, \ldots, \widehat{t}_m)$ is a transposition factorization of the long cycle $\widehat{c} := \pi_1(\til{c}) = (1 \; 2 \cdots n)$.  It is not difficult to show\footnote{For example, it follows from \cite[Cor.~1.4 and Cor.~5.5]{LReiner}; but one can also derive it directly from the tree representation of factorizations in $\Symm_n$ \cite{Denes}.} that for any transposition $\widehat{t}$, there is a factorization $\widehat{\tt}'$ in the Hurwitz orbit of $\widehat{\tt}$ in which $\widehat{t}$ is the first factor.  By choosing a braid $\beta$ such that $\beta(\widehat{\tt}) = \widehat{\tt}'$ and applying it to the last $m - k$ coordinates of $\til{\tt}$, we may take $\til{t}_{k + 1}$ to have nonzero entries in any pair of off-diagonal positions that we like.  Consequently, we may apply the following sort of Hurwitz moves to arrange one diagonal reflection to have non-$1$ entry in the $(1, 1)$ position:
\[
\left(
\begin{bmatrix}
1 & 0 \\
0 & b
\end{bmatrix}
,
\begin{bmatrix}
0      & a \\
a^{-1} & 0
\end{bmatrix}
\right)
\overset{\sigma}{\longrightarrow}
\left(
\begin{bmatrix}
0      & a \\
a^{-1} & 0
\end{bmatrix}
,
\begin{bmatrix}
b & 0 \\
0 & 1
\end{bmatrix}
\right)
\overset{\sigma}{\longrightarrow}
\left(
\begin{bmatrix}
b & 0 \\
0 & 1
\end{bmatrix}
,
\begin{bmatrix}
0      & ab^{-1} \\
a^{-1}b & 0
\end{bmatrix}
\right).
\]
Since all diagonal reflections commute with each other, we may successively move each one into the $k$th position and apply the same procedure to give a factorization in the same Hurwitz orbit as $\tt$ in which all diagonal reflections have their unique nonzero weight in the first position.

The product of the transposition-like factors in $\tt$ belongs to the subgroup $G(\infty, \infty, n)$ of weight-$0$ elements.  Consequently, the sum of the weights of the diagonal factors must be equal to the weight of $\til{c}$, which is $1$.  This implies that, after performing the Hurwitz moves above, the product of the diagonal reflections is $[\id; (1, 0, \ldots, 0)]$ and the product of the transposition-like factors is the permutation matrix $\widehat{c}$.

We continue to focus on the suffix consisting of transposition-like factors.  By \cite[Thm.~1.1]{LReiner} (which has the same statement as our Main Theorem but in the case of finite real reflection groups) applied in $\Symm_n$, the factorization $\widehat{\tt} = (\widehat{t}_{k + 1}, \ldots, \widehat{t}_{m})$ has in its Hurwitz orbit a factorization in which the first $m - k - (n - 1)$ factors are all equal to the transposition $(1 \; 2)$ and the last $n - 1$ factors are $(1\; 2)$, $(2\; 3)$, \ldots, $(n - 1\; n)$, a minimal factorization of $\widehat{c}$.  (Incidentally, this implies that $m - k - (n - 1)$ is even.)  Choose a braid $\beta$ that has this effect on $\widehat{\tt}$; then 
\begin{multline*}
\beta( t_{k + 1}, \ldots, t_m )
=
\Big( 
\left[(1\; 2); a_{k + 1}\right],
\;
\ldots ,
\;
\left[(1\; 2); a_{m - n + 1}\right],
\\
\left[(1\; 2); a_{m - n + 2}\right],
\;
\left[(2\; 3); a_{m - n + 3)}\right],
\;
\ldots ,
\;
\left[(n - 1\; n); a_{m}\right]
\Big)
\end{multline*}
for some integers $a_{k + 1}, \ldots, a_m$.  Because everything is so explicit, by carrying out the multiplication we can immediately read off two facts: first, that $a_{m - n + 3} = \dots = a_{m} = 0$, and second, that
\begin{equation}
\label{infinite dihedral factorization}
\left[(1\; 2); a_{k + 1}\right]
\cdots 
\left[(1\; 2); a_{m - n + 2}\right]
=
[(1\; 2); 0].
\end{equation}
The subgroup of $G_\infty$ generated by the factors $[(1\; 2); a_i]$ for $k + 1 \leq i \leq m - n + 2$ is isomorphic to a subgroup of the infinite dihedral group $G(\infty, \infty, 2)$.  Therefore, applying the results described in Remark~\ref{rmk:LR} to \eqref{infinite dihedral factorization}, this subfactorization has in its Hurwitz orbit a factorization of the form
\begin{multline*}
\Big(
\left[(1\; 2); b_{k + 1}\right], \left[(1 \; 2); b_{k + 1}\right], 
\left[(1\; 2); b_{k + 3}\right], \left[(1\; 2); b_{k + 3}\right],
\ldots \\ \ldots, 
\left[(1\; 2); b_{m - n)}\right], \left[(1\; 2); b_{m - n)}\right],
[(1\; 2); 0] \Big) 
\end{multline*}
for some integers $b_{k + 1}, b_{k + 3}, \ldots, b_{m - n}$.

We now summarize our progress so far: given an arbitrary reflection factorization $\tt$ of the element $c$ in $G_d$ consisting of $k$ diagonal reflections and $m - k$ transposition-like reflections, we have selected a covering factorization $\til{\tt}$ of $\til{c}$ in $G_\infty$ and shown that $\til{\tt}$ is in the same Hurwitz orbit as a factorization $(\til{r}_1, \ldots, \til{r}_m)$ with the following properties:
\begin{enumerate}
\item $\til{r}_1, \ldots, \til{r}_k$ are diagonal reflections with nonzero weight in the first coordinate and product $[\id, (1, 0, \ldots, 0)]$;
\item there are integers $b_{k + 1}, b_{k + 3}, \ldots, b_{m - n}$ such that $\til{r}_{k + 1} = \til{r}_{k + 2} = [(1 \; 2); b_{k + 1}]$, $\til{r}_{k + 3} = \til{r}_{k + 4} = [(1 \; 2); b_{k + 3}]$, \ldots, $\til{r}_{m - n} = \til{r}_{m - n + 1} = [(1 \; 2); b_{m - n}]$; and
\item $\til{r}_{m - n + 2} = [(1 \; 2); 0]$, $\til{r}_{m - n + 3} = [(2 \; 3); 0]$, \ldots, $\til{r}_{m} = [(n - 1 \; n); 0]$.  
\end{enumerate}

In the next stage of the argument, we \emph{cable} the diagonal factors together, in the following sense.  Given a factorization $\boldf = (f_1, \ldots, f_m)$ of an element $g$ in a group $G$, choose an interval $I = [a, b] \subseteq \{1, \ldots, m\}$.  Then the result of cabling $\boldf$ at $I$ is the length-$(m - b + a)$ factorization
\[
(f_1, \ldots, f_{a - 1}, \quad f_a \cdots f_b, \quad f_{b + 1}, \ldots, f_m)
\]
of $g$ in $G$.  The action of the braid group $\BB_{m - b + a}$ on cabled factorizations lifts in a natural way to the action of the $m$-strand braid group on the original factorization: the result
\[
\begin{array}{rccl}
 \sigma_{a - 1} \cdot (f_1, \ldots, f_{a - 2}, &  f_{a - 1}, & f_a \cdots f_b, & f_{b + 1}, \ldots, f_m)
= \\
(f_1, \ldots, f_{a - 2}, &  f_a \cdots f_b, &  (f_{a-1})^{f_a \cdots f_b}, & f_{b + 1}, \ldots, f_m)
\end{array}
\]
of a single Hurwitz move is the cabling of 
\[
\sigma_{b - 1}\cdots\sigma_{a - 1}\cdot \boldf
=
(f_1, \ldots, f_{a - 2}, \quad f_a, \; \ldots, \; f_b, \quad (f_{a-1})^{f_a \cdots f_b}, \quad
f_{b + 1}, \ldots, f_m)
\]
at the interval $[a - 1, b - 1]$,
and similarly
\[
\begin{array}{rccl}
\sigma_a \cdot (f_1, \ldots,  f_{a - 1}, & f_a \cdots f_b, & f_{b + 1}, & f_{b + 2}, \ldots, f_m) = \\
(f_1, \ldots, f_{a - 1}, & f_{b+1}, & (f_a \cdots f_b)^{f_{b+1}}, & f_{b + 2}, \ldots, f_m)
\end{array}
\]
is the cabling of 
\[
\sigma_{a}\cdots\sigma_{b}\cdot \boldf
=
(f_1, \ldots, f_{a - 1}, \quad f_{b + 1}, \; (f_a)^{f_{b + 1}},  \; \ldots, \;  (f_b)^{f_{b + 1}}, \quad f_{b + 2}, \ldots, f_m)
\]
at the interval $[a + 1, b + 1]$.  This is illustrated in Figure~\ref{fig:cabling}.

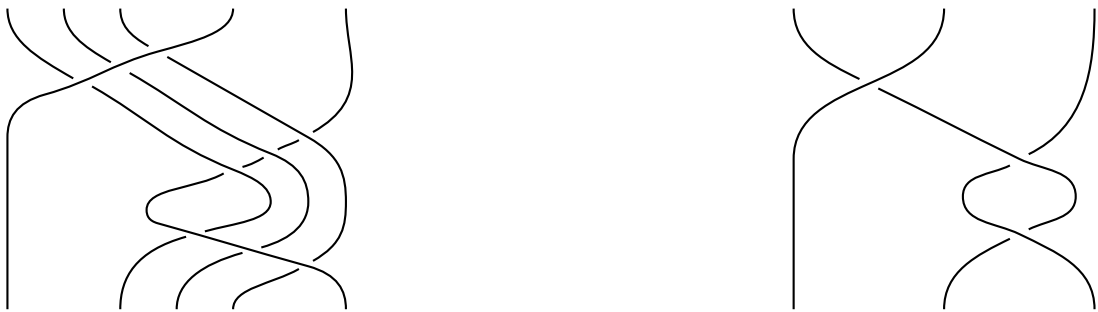
\begin{figure}

\begin{tikzpicture}[yscale=4/14, xscale=.5]
\draw[thick] (6, 14) to [out=-90, in=26] (4, 12) to [out=206, in=26] (1, 10) to [out=206, in=90] (0, 8) to [out=-90, in=90] (0, 0) ;

\draw[thick] (0, 14) to [out=-90, in=135] (1.75, 10.75);
\draw[thick] (2.25, 10.375) to [out=-45, in=145] (6, 6.5) to [out=-35, in=90] (7, 5) to [out=-90, in=26] (5.25, 3.625);
\draw[thick] (4.75, 3.375) to [out=206, in=90] (3, 0);

\draw[thick] (1.5, 14) to [out=-90, in=135] (2.75, 11.5);
\draw[thick] (3.25, 11) to [out=-45, in=145] (7, 7.25) to [out=-35, in=90] (8, 5) to [out=-90, in=26] (6.75, 2.875);
\draw[thick] (6.25, 2.625) to [out=206, in=90] (4.5, 0);

\draw[thick] (3, 14) to [out=-90, in=135] (3.75, 12.25);
\draw[thick] (4.25, 11.75) to [out=-45, in=135] (8, 8) to [out=-45, in=90] (9, 5) to [out=-90, in=45] (8.125, 2.25);
\draw[thick] (7.75, 1.875) to [out=225, in=90] (6, 0);

\draw[thick] (9, 14) to [out=-90, in=45] (8.125, 8.25);
\draw[thick] (7.75, 7.875) to [out=225, in=45] (7.1875, 7.4375);
\draw[thick] (6.8125, 7.0625) to [out=225, in=26] (6.25, 6.625);
\draw[thick] (5.75, 6.325) to [out=225, in=90] (3.7, 4.625) to [out=-90, in=154] (4, 4) to [out=-26, in = 154] (8, 2) to [out=-26, in=90] (9, 0);
\end{tikzpicture}
\hfill
\begin{tikzpicture}[scale=.5]
\draw[thick] (4, 8) to [out=-90, in=90] (0, 4) to [out=-90, in=90] (0, 0) ;
%\draw[thick, rounded corners] (4, 8) -- (4, 7) -- (0, 5) -- (0, 0) ;
\draw[thick] (0, 8) to [out=-90, in=154] (1.75, 6.125);
\draw[thick] (2.25, 5.875) to [out=-26, in=154] (6, 4) to [out=-26, in=90] (7.5, 3) to [out=-90, in=26] (6.25, 2.125);
\draw[thick] (5.75, 1.875) to [out=206, in=90] (4, 0);
%\draw[thick, rounded corners] (0, 8) -- (0, 7) -- (8, 3) -- (4, 1) -- (4, 0) ;
\draw[thick] (8, 8) to [out=-90, in=26] (6.25, 4.125);
\draw[thick] (5.75, 3.825) to [out=206, in=90] (4.5, 3) to [out=-90, in=154] (6, 2) to [out=-26, in=90] (8, 0);
%\draw[thick, rounded corners] (8, 8) -- (8, 5) -- (4, 3) -- (8, 1) -- (8, 0) ;
\end{tikzpicture}

\caption{Cabling: Applying $(\sigma_4^{-1}\sigma_3^{-1}\sigma_2^{-1})(\sigma_2^{-1}\sigma_3^{-1}\sigma_4^{-1})(\sigma_1\sigma_2\sigma_3)$ to the factorization $(a, b, c, d, e)$ of $g = abcde$ (left) produces 
%$(d, a^{abcde^{-1} d^{-1}c^{-1}b^{-1}a^{-1}d}, b^{abcde^{-1} d^{-1}c^{-1}b^{-1}a^{-1}d}, c^{abcde^{-1} d^{-1}c^{-1}b^{-1}a^{-1}d}, e^{ d^{-1}c^{-1}b^{-1} a^{-1} d})$
$(d, a^{ge^{-1} g^{-1}d}, b^{ge^{-1} g^{-1}d}, c^{ge^{-1} g^{-1}d}, e^{ g^{-1} d})$.  Applying the braid $(\sigma_2^{-1})(\sigma_2^{-1})(\sigma_1)$ (right) to the cabled factorization $(abc, d, e)$ produces  $(d, (abc)^{g e^{-1} g^{-1}d} , e^{g^{-1} d})$.
}
\label{fig:cabling}
\end{figure}

Next, we use the cabling of the diagonal factors, together with many repetitions of the moves
\begin{align*}
\left(
\begin{bmatrix}
z & 0 \\
0 & 1
\end{bmatrix},
\begin{bmatrix}
0 & z^b \\
z^{-b} & 0
\end{bmatrix},
\begin{bmatrix}
0 & z^b \\
z^{-b} & 0
\end{bmatrix}
\right)
& \overset{\sigma_2 \sigma_1}{\longrightarrow}
\left(
\begin{bmatrix}
0 & z^b \\
z^{-b} & 0
\end{bmatrix},
\begin{bmatrix}
0 & z^b \\
z^{-b} & 0
\end{bmatrix},
\begin{bmatrix}
z & 0 \\
0 & 1
\end{bmatrix}
\right) \\
& \overset{\sigma_1 \sigma_2}{\longrightarrow}
\left(
\begin{bmatrix}
z & 0 \\
0 & 1
\end{bmatrix},
\begin{bmatrix}
0 & z^{b-1} \\
z^{-b+1} & 0
\end{bmatrix},
\begin{bmatrix}
0 & z^{b-1} \\
z^{-b+1} & 0
\end{bmatrix}
\right),
\end{align*}
to make all of these $[(1\; 2); b]$ factors equal to $[(1\; 2); 0]$: once the first pair has been reduced to $[(1\; 2); 0]$, they can be cabled together (with product the identity) and moved to the end of the middle section of the factorization, allowing the next pair to be reduced.  

At this point, we have that $\til{\tt}$ has in its Hurwitz orbit a factorization 
\[
(\til{r}_1, \ldots, \til{r}_k, [(1\; 2), 0], [(1\; 2), 0], \ldots, [(1\; 2), 0], [(2\; 3), 0], \ldots [(n - 1\; n), 0])
\]
where $\til{r}_1, \ldots, \til{r}_k$ are diagonal reflections with nonzero weight in the first coordinate.
Finally, we use Hurwitz moves to permute the diagonal factors (which all commute with each other) into the following order: we first place all the reflections whose weight is congruent to $1$ modulo $d$ in order from smallest weight to largest, followed by those whose weight is congruent to $2$ modulo $d$ in order from smallest weight to largest, and so on.  The resulting factorization obviously depends only on the weights of $\til{r}_1, \ldots, \til{r}_k$, or equivalently only on the weights of the reflections in $\til{\tt}$.  Moreover, the image of this factorization under $\pi_d$ is in the same Hurwitz orbit as $\tt$ and is uniquely determined by the multiset of conjugacy classes of reflections in $\tt$.  Thus every factorization with the same multiset of conjugacy classes belongs to the same Hurwitz orbit as $\tt$.  This completes the proof.

\section{Exceptional groups}
\label{sec:exceptionals}

As mentioned in the introduction, the present author and Reiner have conjectured \cite[Conj.~6.3]{LReiner} that the Hurwitz action is ``as transitive as possible'' on reflection factorizations of a Coxeter element in any well generated complex reflection group.  Ideally, one would hope for a uniform proof of this statement.  However, even in the case of shortest factorizations (Theorem~\ref{thm:bessis} of Bessis), the only known proofs are case-by-case.  Below, we discuss the situation in more detail.

There are $34$ irreducible finite complex reflection groups not contained in the infinite family; in the Shephard--Todd classification, they are named $G_4, G_5, \ldots, G_{37}$.  Of the exceptional groups, $26$ are well-generated, including the six exceptional real reflection groups, of types $\type{H}_3$, $\type{H}_4$, $\type{F}_4$, $\type{E}_6$, $\type{E}_7$, and $\type{E}_8$ (respectively $G_{23}$, $G_{30}$, $G_{28}$, $G_{35}$, $G_{36}$, and $G_{37}$).  Thus, there are $20$ groups for which the question considered here makes sense.

In \cite{Zach} and \cite{Lazreq}, two groups of authors considered the smallest of the well generated exceptional groups, namely $G_4$, $G_5$, and $G_6$.  They proved in each case that the analogue of our Main Theorem is true, that is, that two reflection factorizations of a Coxeter element in one of these groups lie in the same Hurwitz orbit if and only if they have the same multiset of conjugacy classes.  The proofs in all cases are inductive: it is shown that for factorizations involving sufficiently many (say, $k$), factors, one can use an approach similar to the ``cabling'' strategy above to reduce the problem to considering factorizations of length $k - 1$, and then the result is established by exhaustive computation for short factorizations.  In principle, the same approach (particularly in the form used by Lazreq et al.) should work for all of the remaining groups; however, in practice, there is a huge gap between naive bounds for when the inductive step applies and what base cases are computationally feasible to check, even for groups of rank~$2$.  

Other approaches may be possible.  The proof of Theorem~\ref{thm:WY} in \cite{WY} is uniform, but makes heavy use of Coxeter-specific tools (the Coxeter length and Bruhat order) that do not have good analogues in the complex case.  In \cite{LReiner}, the main tool used was a lemma concerning the possible acute angles among the roots in a circuit (a minimal linearly dependent set) in a real root system.  It is conceivable that such an approach could be coupled with the techniques of \cite{Zach, Lazreq} to give a proof in at least some complex groups; but the root circuit lemma in \cite{LReiner} is ultimately proved via a brute-force computational attack even in Weyl groups.

Separately, one might hope to extend the investigation to factorizations of elements other than Coxeter elements.  There are three invariants attached to a tuple $(t_1, \ldots, t_m)$ of reflections in a complex reflection group $G$ that are easily seen to be preserved by the Hurwitz action:
\begin{itemize}
\item the product $t_1 \cdots t_m$ of the $t_i$;
\item the subgroup $H = \langle t_1, \ldots, t_m \rangle$ of $G$ generated by the $t_i$; and
\item the multiset $\big\{ \{ h t_i h^{-1} \colon h \in H\} \colon i = 1, \ldots, m\big\}$ of orbits of the $t_i$ under conjugation by $H$.
\end{itemize}
(For factorizations of a Coxeter element, the subgroup $H$ is always the full group $G$.)  In \cite{BergerDihedral}, it was shown that these invariants distinguish Hurwitz orbits when $G$ is a dihedral group.  One is tempted to conjecture that the same result is true for reflections in any complex reflection group.  This conjecture is consistent with the known results on Coxeter elements, as well as with the work of Kluitmann \cite{Kluitmann} and Ben-Itzhak--Teicher \cite{BenItzhak-Teicher} in the symmetric group.  Work-in-progress by  Minnick--Pirillo--Racile--Wang and J. Wang (respectively, for tuples of arbitrary length in small exceptional groups, and for tuples that constitute minimum-length factorizations of arbitrary elements in the infinite family; personal communications) also lend credence to the conjecture.

\section*{Acknowledgements}
The author is grateful to an anonymous referee for helpful comments.
This research was supported in part by an ORAU Powe Award and a Simons Collaboration Grant (634530).

	\bibliographystyle{amsalpha}
	\bibliography{Hurwitz-nomr}

\end{document}